\documentclass[11pt,a4paper]{article}

\usepackage{theorem,enumerate}
\usepackage{amsmath,latexsym,amssymb,amsfonts}
\usepackage{eucal}
\usepackage{mathrsfs}
\usepackage{array}
\usepackage{footmisc}
\usepackage{graphicx, color}
\usepackage{caption,subcaption}

\newtheorem{theorem}{Theorem} \rm
\newtheorem{lemma}[theorem]{Lemma}

\newtheorem{guess}[theorem]{Conjecture}

\newtheorem{question}[theorem]{Question}
\numberwithin{theorem}{section}
\usepackage{color}

\newcommand{\qed}{\hfill \ensuremath{\Box}}

\begin{document}
\title{\bf Generalized list colouring of graphs}

\author{\small Eun-Kyung Cho\thanks{
		Department of Mathematics, Hankuk University of Foreign Studies, Yongin-si, Gyeonggi-do, Republic of Korea. \texttt{ourearth3@gmail.com}
	},  \  \
	\small Ilkyoo Choi\thanks{
		Department of Mathematics, Hankuk University of Foreign Studies, Yongin-si, Gyeonggi-do, Republic of Korea.
		\texttt{ilkyoo@hufs.ac.kr}
	},  \  \
		\small Yiting Jiang\thanks{Department of Mathematics, Zhejiang Normal University, China. 
		\texttt{ ytjiang@zjnu.edu.cn		}
		},  \  \
	\small Ringi Kim\thanks{
		Department of Mathematical Sciences, KAIST, Daejeon, Republic of Korea.
		\texttt{kimrg@kaist.ac.kr}},  \  \
	\small Boram Park\thanks{
		Department of Mathematics, Ajou University, Suwon-si, Gyeonggi-do, Republic of Korea.
		\texttt{borampark@ajou.ac.kr}}
	,  \  \
	\small Jiayan Yan\thanks{Department of Mathematics, Zhejiang Normal University, China. 
		\texttt{ yanjiayan19950403@vip.qq.com	}
	},  \  \
	\small Xuding Zhu\thanks{
		Department of Mathematics, Zhejiang Normal University, China.
		\texttt{xdzhu@zjnu.edu.cn}
	}}
	\date\today
	\maketitle

\author{ 
	}


\begin{abstract}
This paper disproves a conjecture in [ Wang,  Wu,   Yan and   Xie, 
  A Weaker Version of a Conjecture on List Vertex
	Arboricity of Graphs,  
Graphs and Combinatorics (2015) 31:1779–1787]  and answers in negative a question in [ Dvo\v{r}\'{a}k,   Pek\'{a}rek  and   Sereni,   On generalized choice and coloring numbers,  arXiv: 1081.0682403, 2019].
In return, we pose five open problems. 
\end{abstract}

\section{Introduction}
  
 Assume ${\cal G}$ is a hereditary family of  graphs, i.e., if $G \in {\cal G}$ and $H$ is an induced subgraph of $G$, then $H \in {\cal G}$. A ${\cal G}$-colouring of a graph $G$ is a colouring $\phi$ of the vertices of $G$   so that each colour class induces  a graph in ${\cal G}$.
 A  ${\cal G}$-$n$-colouring of   $G$ is a ${\cal G}$-colouring $\phi$ of $G$ such that $\phi(v) \in [n] = \{1,2,\ldots, n\}$ for each vertex $v$. We say $G$ is ${\cal G}$-$n$-colourable if there exists a 
 ${\cal G}$-$n$-colouring of  $G$.
 The ${\cal G}$-chromatic number of $G$ is 
 $$\chi_{\cal G}(G) = \min \{n: G \text{ is  ${\cal G}$-$n$-colourable}\}.$$
 
 Assume $L$ is a list assignment of $G$. A ${\cal G}$-$L$-colouring of $G$ is a ${\cal G}$-colouring $\phi$ of $G$ so that $\phi(v) \in L(v)$ for each vertex $v$. We say $G$ is ${\cal G}$-$n$-choosable if 
 for every $n$-list assignment $L$ of $G$, there exists a  ${\cal G}$-$L$-colouring of $G$, The {\em ${\cal G}$-choice number} of $G$ is 
 $$ch_{\cal G}(G) = \min \{n: G \text{ is ${\cal G}$-$n$-choosable}\}.$$
 
The concept of  ${\cal G}$-colouring of a graph is a slight modification of the concept of generalized colouring of graphs introduced in \cite{DPS}, where the graph class ${\cal G}$ is assumed to be of the form ${\cal G} = \{G : f(G) \le d\}$ for some graph parameter $f$ and constant $d$. We find that there are some graph families ${\cal G}$ for which 
the ${\cal G}$-colouring problems are interesting, and yet ${\cal G}$ is not easily expressed in such a form.

Many colouring concepts studied in the literature are 
${\cal G}$-colourings for special graph families ${\cal G}$.

We denote by 
\begin{itemize}
	\item ${\cal G}_k$ the family of   graphs whose connected components are of order at most $k$;
	\item ${\cal D}_k$  the family of   graphs   of maximum degree at most $k$;
		\item ${\cal F}$  the family of  forests;
				\item ${\cal S}$  the family of star forests;
					\item ${\cal L}$  the family of linear forests;
						\item ${\cal C}_k$  the family of graphs of colouring number at most $k$.
							\item ${\cal M}_k$  the family of graphs of maximum average degree at most $k$.
\end{itemize}

Many of the ${\cal G}$-colourings have special names and are studied extensively in the literature.
\begin{itemize}
	\item A   ${\cal G}_k$-colouring of $G$ is a   colouring of $G$ with clustering $k$. In particular,  a ${\cal G}_1$-colouring of $G$ is a proper colouring of $G$.
	\item A   ${\cal D}_k$-colouring of $G$ is a $k$-defective   colouring of $G$. The parameter $\chi_{{\cal D}_k}(G)$ is the \emph{$k$-defective chromatic number} of $G$. Also,
	a ${\cal D}_0$-colouring of $G$ is a proper colouring of $G$.
	\item An   ${\cal F}$-colouring of $G$ is a vertex arboreal    colouring of $G$. The parameter $\chi_{\cal F}(G)$ is the \emph{vertex arboricity} of $G$, and $ch_{\cal F}(G)$ is the \emph{list vertex arboricity} of $G$.
	\item  The parameter $\chi_{\cal S}(G)$ is the \emph{star vertex arboricity} of $G$, and $ch_{\cal S}(G)$ is the \emph{star list vertex arboricity} of $G$.
	\item  The parameter $\chi_{\cal L}(G)$ is the \emph{linear vertex arboricity} of $G$, and $ch_{\cal L}(G)$ is the \emph{linear  list vertex arboricity} of $G$.
\end{itemize}

For  any two graph families  ${\cal G}$ and ${\cal G}'$, for any graph $G$, it follows easily from the definition that 
\begin{eqnarray}
\label{eq1}
 \chi_{\cal G}(G) &\le& (\max_{H \in {\cal G}'}\chi_{\cal G}(H)) \chi_{{\cal G}'}(G),
\end{eqnarray} 
and this upper bound is tight.
For example, $$\chi(G) \le 2 \chi_{\cal F}(G),    \  \chi_{\cal S}(G) \le 2 \chi_{\cal F}(G)    \text{  and } \chi(G) \le (k+1) \chi_{{\cal M}_k}(G),$$ and    for any integers  $k,k'$, $$\chi_{{\cal G}_k}(G) \le \left\lceil \frac{k'}{k} \right\rceil \chi_{{\cal G}_{k'}}(G),$$
 and equalities hold for some graphs $G$.

It is natural to ask if the same or similar inequalities hold for the corresponding choice number. Some of such inequalities are posed as conjectures or questions in the literature.
 For example, the following 
  conjecture was proposed in \cite{WWYX}:

\begin{guess}
	\label{g1}
	For any graph $G$, 
	\begin{eqnarray*}
	ch(G) &\le & 2 ch_{\cal F}(G).  
	\end{eqnarray*}  
\end{guess}
The following question  was asked in \cite{DPS}:

\begin{question}
	\label{q1}
	Is it true that for any graph $G$, for any positive integer $k$, 
	$$ch(G) \le (k+1) ch_{{\cal M}_k}(G)?$$
\end{question}

  In this note, 
we disprove Conjecture \ref{g1} and give a negative answer to Question \ref{q1}.  

\section{The proofs}

\begin{lemma}
	\label{lem-star}
	Assume   $k\ge 2$  and $m = k(k+1)-1$ are integers. Then for any positive integer $n$,  $ch_{\cal S}(K_{m,n}) \le k$.
\end{lemma}
\begin{proof}
	Assume $k, n \ge 2$ are integers and $m = k(k+1)-1$.
	Let $G=K_{m, n}$ be the complete bipartite graph with partite sets $A, B$, with 
	$|A|=m$ and $B = n$. We show that $ch_{\cal S}(G) \le k$. 
	
	Let $L$ be a $k$-list assignment of $G$. Build a bipartite graph $H$ with partite sets $A$ and $C = \cup_{v \in A} L(v)$, and in which $vc$ is an edge if and only if $c \in L(v)$.
	Note that each vertex $v \in A$ has degree $k$ in $H$.
	
A subset $C'$ of $C$ is  {\em heavy} if $|N_H(C')| \ge (k+1)|C'|$.
In particular, $\emptyset$ is a heavy subset of $C$. 	
Let $C'$ be a maximal heavy subset of $C$. Let $A' = N_H(C')$ and 
$H' = H-(A' \cup C')$.

	Then each vertex $v \in A-A'$ has degree $k$ in $H'$.
		If there is a colour $c $ for which  $d_{H'}(c ) \ge k+1$, then let 
	  $$C'' =C' \cup \{c\}.$$ 
	Then  $  |N_H(C'')| = |N_H(C')|+d_{H'}(c)  \ge (k+1)|C''|$. So $C''$ is heavy, contrary to our assumption that $C'$ is a maximum heavy subset of $C$. 
	 
	So each vertex $c  \in C - C'$ has degree at most $k$ in $H'$. 
	By Hall's Theorem,  there is a matching $M$ in $H'$ that covers all the vertices of $A- A'$. Let $\phi$ be the $L$-colouring of $A-A'$ defined as $\phi(v) = c$ if $vc \in M$. So all vertices of $A-A'$ are coloured by distinct colours. Extend $\phi$ to an $L$-colouring of $H$ as follows:
	\begin{itemize}
		\item Since $k(k+1) > |A| \ge |N_{H}(C')| \ge |C'|(k+1)$,  we know that $|C'|   \le k-1$.  For each vertex  $v \in B$, we have $L(v) - C' \ne \emptyset$. Let $\phi(v)$ be any colour in $L(v) - C'$. 
		\item For each vertex $v \in A'$, as $A' = N_H(C')$,   $L(v) \cap C'\ne \emptyset$. Let $\phi(v)$ be any colour in $L(v) \cap C'$.  
	\end{itemize}

	This is an ${\cal S}$-$L$-colouring of $G$,  as each connected monochromatic subgraph of $G$ contains  at most one vertex of $A$, and hence is a star. This completes the proof of Lemma \ref{lem-star}. \qed
	\end{proof}
	
	It is well-known that if $n \ge m^m$, then $ch(K_{m,n}) = m+1$.
	The following lemma shows that for any constant $d$, if $n$ is sufficiently large, then 
	$ch_{{\cal D}_d}(K_{m,n}) = m+1$.
	
	\begin{lemma}
		\label{lem-dd}
		Assume $d$ is a non-negative integer. If $n \ge (dm+1)m^m$, then
			$ch_{{\cal D}_d}(K_{m,n}) = m+1$. 
	\end{lemma}
	\begin{proof}
		Assume $n \ge (dm+1)m^m$ and $G=K_{m,n}$ with partite sets $A, B$, where $|A|=m$ and $|B|=n$. As $G$ is $m$-degenerate, we have $ch_{{\cal D}_d}(G) \le ch(G) \le m+1$.
		
		Now we show that $ch_{{\cal D}_d}(G) > m$.
		
		Let $L$ be the   $m$-assignment which assigns to vertices $v \in A$ pairwise disjoint $m$-sets $\{L(v): v \in A\}$. 
		Let $\Phi$ be the set of all   $L$-colourings $\phi$ of $A$. 
		Thus $|\Phi| = m^m$. For each $\phi \in \Phi$, assign a $(dm+1)$-subset $B_{\phi}$ of $B$ so that for distinct $\phi, \phi' \in \Phi$, $B_{\phi} \cap B_{\phi'} = \emptyset$. 
		Since $|B| \ge (dm+1)m^m$, such an assignment exists.
		Extend $L$ to an $m$-assignment of $G$ by letting 
		$L(v) = \phi(A)$ for any $v \in B_{\phi}$. Assign arbitrary $m$ colours to $v$ if $v \in B$ is not contained in any subsets $B_{\phi}$. 
		
		Now we show that $G$ is not ${\cal D}_d$-$L$-colourable.
		Assume to the contrary that $\phi$ is a ${\cal D}_d$-$L$-colouring of $G$. Let $\phi|_A$ be the restriction of $\phi$ to $A$. 
		For any $v \in B_{\phi|_A}$, $\phi(v) \in L(v) = \phi(A)$. As
		$|\phi(A)|=m$ and $|B_{\phi|_A}| =(dm+1)$, there exists a colour $c \in \phi(A)$ such that $|\phi^{-1}(c) \cap B_{\phi|_A}| \ge d+1$. Assume $u \in A$ and $c=\phi(u)$. Then $u$ has at least $d+1$ neighbours that are coloured the same colour as $u$ itself. So $\phi$ is not a  ${\cal D}_d$-$L$-colouring of $G$.
		
		This completes the proof of Lemma \ref{lem-dd}.\qed
		\end{proof}
	
	As a corollary of Lemmas \ref{lem-star} and \ref{lem-dd}, we have the following theorem.
	
\begin{theorem}
	\label{thm-main}
	 For any integers $k, d$ with $k \ge 2$, there exists a graph $G$ with $ch_{\cal S}(G) \le k$ and $ch_{{\cal D}_d}(G) = k(k+1)$. 
	 In particular, for any constant $p$, there exists a graph $G$ with 
	 $$ch(G) \ge p \cdot ch_{\cal S}(G) \ge p \cdot ch_{\cal F}(G) \ge p \cdot ch_{{\cal M}_2}(G).$$
\end{theorem} 

This theorem refutes Conjecture \ref{g1} and gives a negative answer to Question \ref{q1}. We remark that Conjecture \ref{g1},    posed at the end of   \cite{WWYX},   is not the conjecture referred to   in the title of that paper. The main conjecture studied in \cite{WWYX} is the following conjecture posed in \cite{ZW}:

\begin{guess}
	\label{g2}
	If $|V(G)| \le 3 \chi_{\cal F}(G)$, then $ch_{\cal F}(G) = \chi_{\cal F}(G)$.
\end{guess}

This conjecture remains open. 

It is known \cite{DPS} that $ch(G)$ is bounded from above by a function of $ch_{\cal G}(G)$, provided that graphs in ${\cal G}$ have bounded maximum average degree. Or equivalently, graphs in $G$ have bounded choice number. In particular, $ch(G) \le f(ch_{\cal F}(G))$ for some function $f$. The function  $f$ found in \cite{DPS} is exponential.   
 Theorem \ref{thm-main} shows that $f$ cannot be a linear function. It would be interesting to know if there is a polynomial function $f$
 such that  $ch(G) \le f(ch_{\cal F}(G))$.

\begin{question}
\label{nq1}
Are there constant integers $a,b$ such that 
$$ch(G) \le a (ch_{\cal F}(G))^b?$$
If so, what is the smallest such integer $b$? 
\end{question}

It would also be interesting to know if the bound given in Theorem \ref{thm-main} is tight. I.e., is it true that $ch(G) \le ch_{\cal F}(G) ( ch_{\cal F}(G)+1)$ for all graphs $G$? 

As observed in the introduction, for any two graph classes ${\cal G}$ and ${\cal G}'$, $\chi_{\cal G}(G)  \le  (\max_{H \in {\cal G}'}\chi_{\cal G}(H)) \chi_{{\cal G}'}(G)$.  We are interested in the question whether the same inequality holds for the corresponding choice number. 
If  ${\cal G}'  \subseteq {\cal G}$, then trivially, the inequality $ch_{\cal G}(G)  \le  (\max_{H \in {\cal G}'} ch_{\cal G}(H)) ch_{{\cal G}'}(G) = ch_{{\cal G}'}(G)$ holds.
We do now know any non-trivial case where the inequality    $ch_{\cal G}(G)  \le  (\max_{H \in {\cal G}'} ch_{\cal G}(H)) ch_{{\cal G}'}(G)$ holds. 
As remarked in \cite{DPS}, the following question may have a positive answer.  

\begin{question} \cite{DPS}
	\label{q2} Is it true that for any graph $G$ and any positive integer $k$, $$ch(G) \le k ch_{{\cal G}_k}(G)?$$
\end{question}

Even the $k=2$ case of the above question is very interesting and challenging.   
  More generally, the following question seems to be natural and interesting:
 
 \begin{question}
 	\label{q3} Is it true that for any graph $G$ and any positive integers $k, k'$, $$ch_{{\cal G}_{k'}} (G) \le \left\lceil \frac k{k'} \right\rceil ch_{{\cal G}_k}(G)?$$
 \end{question}
 
 The relation between $ch_{\cal L}(G)$ and $ch_{\cal S}(G)$ is also interesting. By Theorem \ref{thm-main}, there are graphs $G$ for which 
 $$ch_{\cal L}(G) \ge ch_{{\cal D}_2}(G) \ge ch_{\cal S}(G) (ch_{\cal S}(G)+1).$$
 It follows from (\ref{eq1}) that 
 $$\chi_{\cal S}(G) \le 2 \chi_{\cal L}(G).$$
 The following questions remain open.

 \begin{question}
 	\label{q4} Is it true that for any graph $G$, $$ch_{\cal S} (G) \le 2  ch_{{\cal L}}(G)?$$
 \end{question}

 \begin{question}
 	\label{q4} Is it true that for any graph $G$, $$ch_{\cal L} (G) \le   ch_{{\cal S}}(G) (ch_{\cal S}(G)+1)?$$
 	Or is there an integer $a$ such that
 	$$ch_{\cal L} (G) \le   (ch_{\cal S}(G))^{a}?$$
 \end{question}

\end{document}